\documentclass{amsart}[11pt]
\usepackage{amssymb}
\usepackage{tabmac}

\usepackage[margin=1.25in]{geometry}

\newtheorem{exercise}{Exercise}
\newtheorem{lem}[exercise]{Lemma}
\newtheorem{prop}[exercise]{Proposition}
\newtheorem{proposition}[exercise]{Proposition}
\newtheorem{theorem}[exercise]{Theorem}

\theoremstyle{definition}
\newtheorem{example}{Example}
\newtheorem{remark}{Remark}

\newcommand{\af}{\mathrm{af}}
\newcommand{\Aut}{\mathrm{Aut}}
\def\C{{\mathbb C}}

\def\des{\mathrm{des}}
\newcommand{\Des}{\mathrm{Des}}
\newcommand{\ex}{\mathrm{e}}

\newcommand{\Fl}{\mathrm{Fl}}

\newcommand{\geh}{{\mathfrak{g}}}
\newcommand{\Gr}{\mathrm{Gr}}
\newcommand{\hS}{\hat S}
\newcommand{\id}{\mathrm{id}}
\newcommand{\ip}[2]{\langle #1\,,\,#2 \rangle}
\newcommand{\la}{\lambda}

\def\Schub{{\mathfrak{S}}}
\newcommand{\tp}{\tilde{p}}
\newcommand{\tq}{\tilde{q}}
\newcommand{\tP}{\tilde{P}}
\newcommand{\tQ}{\tilde{Q}}
\newcommand{\tS}{\tilde{S}}
\def\Z{{\mathbb Z}}

\title{From quantum Schubert polynomials to $k$-Schur functions via the Toda lattice}
\author{Thomas Lam}
\address{Department of Mathematics,
University of Michigan, 530 Church St., Ann Arbor, MI 48109 USA}
\email{tfylam@umich.edu}
 \thanks{T.L. was supported by NSF grant DMS-0901111, and by a Sloan Fellowship.}
\author{Mark Shimozono}
\address{Department of Mathematics, Virginia Polytechnic Institute
and State University, Blacksburg, VA 24061-0123 USA}
\email{mshimo@vt.edu}
\thanks{M.S. was supported by NSF DMS-0652641 and DMS-0652648.}
\begin{document}
\begin{abstract}
We show that Lapointe-Lascoux-Morse $k$-Schur functions (at $t=1$) and Fomin-Gelfand-Postnikov quantum Schubert polynomials can be obtained from each other by a rational substitution.  This is based upon Kostant's solution of the Toda lattice and Peterson's work on quantum Schubert calculus.
\end{abstract}

\maketitle

\section{The theorem}

\subsection{Quantum Schubert polynomials}
Fomin, Gelfand,  and Postnikov's quantum Schubert polynomials $\Schub_w^q$ are a family of polynomials in the variables $x_1, x_2,\dotsc, x_n$ and the quantum parameters $q_1,q_2,\ldots,q_{n-1}$ indexed by permutations $w \in S_n$.  They show \cite{FGP} that quantum Schubert polynomials represent quantum Schubert classes in the Givental-Kim presentation \cite{GK} of the small quantum cohomology ring $QH^*(\Fl_n)$ of the flag manifold.

\begin{example}
Let $n = 3$.  Then
\begin{align*}
\Schub_1^q &= 1 \hspace{50pt} \Schub_{s_1}^q = x_1 \hspace{42pt} \Schub _{s_2}^q = x_1 + x_2 \\
\Schub_{s_1 s_2}^q &= x_1 x_2 + q \ \ \ \Schub_{s_2s_1}^q = x_1^2 - q \ \ \
\Schub_{s_1s_2s_1}^q = x_1^2x_2 + q_1x_1
\end{align*}
\end{example}

\subsection{$k$-Schur functions}
Let $k = n-1$.  Lapointe, Lascoux, and Morse's $k$-Schur functions \cite{LLM} $s_\lambda^{(k)}$ are a basis of the
Hopf subalgebra $\Lambda_{(n)} = \C[h_1,h_2,\ldots,h_{n-1}] \subset \Lambda$ of symmetric functions generated by the first $n-1$ homogeneous symmetric functions.  They are indexed by $k$-bounded partitions, that is, partitions $\la$ with $\la_1 \leq k$.  Lam \cite{L} showed that $k$-Schur functions represent affine Schubert classes under the realization of the
homology ring $H_*(\Gr)\cong \Lambda_{(n)}$ of the affine Grassmannian $\Gr = \Gr_{SL(n)}$ inside symmetric functions.

\begin{example}
Let $n = 3 = k+1$.  Then
\begin{align*}
s_{2^a 1^{2b}}^{(2)} &= h_2^a e_2^b \\
s_{2^a 1^{2b+1}}^{(2)} &= h_2^a e_2^b h_1
\end{align*}
where $e_2 = h_1^2 - h_2$ is the second elementary symmetric function.
\end{example}

\subsection{The substitution}
For $i \in [0,n]$, let $R_i$ denote the rectangular partition $i^{n-i}$ having $i$ columns and $n-i$ rows.
Let $R'_i$ denote the partition obtained from $R_i$ by removing the outer corner of the Young diagram of $R_i$.
In the following, let $s_{R_0} = s_{R_n} = 1$ and $s_{R'_0} = s_{R'_n} = 0$.

Define a map $\Phi: \C[x_1,x_2,\ldots,x_n,q_1,q_2,\ldots,q_{n-1}] \to \Lambda_{(n)}[s_{R_1}^{-1},\dotsc,s_{R_{n-1}}^{-1}]$ by
\begin{align*}
x_1 + x_2 + \cdots x_i &\longmapsto \frac{s_{R'_i}}{s_{R_i}}\\
q_i &\longmapsto \frac{s_{R_{i-1}}s_{R_{i+1}}}{s_{R_i}^2}
\end{align*}

\begin{theorem}\label{T:main}
Let $w \in S_n$.  Then
$$
\Phi(\Schub_w^q) = \frac{s_{\lambda(w)}^{(k)}}{\prod_{i \in \Des(w)} s_{R_i}}
$$
where $\lambda(w)$ is a $k$-bounded partition
explicitly described in Section \ref{S:la} and $\Des(w)=\{i\mid ws_i<w\}$
denotes the descent set of $w$.
\end{theorem}
Because of a factorization result (Theorem \ref{T:LM}) of Lapointe and Morse, the above theorem determines $s_\lambda^{(k)}$ for all $\lambda$.

\begin{example}
Take $w = s_1 s_2 s_1$.  Then
\begin{align*} \Phi(\Schub_{s_1s_2s_1}^q)&=\Phi(x_1^2x_2 + q_1x_1) \\
&= \left(\frac{h_1}{e_2}\right)^2\left(\frac{h_1}{h_2}-\frac{h_1}{e_2}\right) + \frac{h_2}{e_2^2}\frac{h_1}{e_2} \\&= \frac{h_1}{e_2h_2}.
\end{align*}
Since $\Des(w) = \{1,2\}$ and $\la(w) = (1)$, this agrees with $s^{(2)}_1 = h_1$.
\end{example}

Let us give an outline of the proof.  One first obtains an abstract isomorphism between localizations of the quantum cohomology $QH^*(\Fl_n)$ of the flag variety and the homology $H_*(\Gr)$ of the affine Grassmannian by the composition of three theorems: (1) the theorem of Kim and of Givental and Kim \cite{GK,Kim} which identifies $QH^*(\Fl_n)$ with the coordinate ring of the nilpotent Toda lattice; (2) a theorem of Ginzburg \cite{Gin} and of Peterson which identifies $H_*(\Gr)$ with the coordinate ring of the centralizer of a principal nilpotent element in $PGL(n)$; and (3) a theorem of Kostant \cite{Kos} which solves the nilpotent Toda lattice.  The substitution $\Phi$ arises in this way.

The quantum and affine Schubert classes are compared using a result of Peterson \cite{P} (also proved in \cite{LS, LL}).  We carry out the combinatorics explicitly to obtain the description of Theorem \ref{T:main}.  The results of Fomin, Gelfand, and Postnikov \cite{FGP} and of Lam \cite{L} allow us to formulate the result explicitly in terms of polynomials.

\begin{remark}
The form of the denominator in Theorem \ref{T:main} probably follows from the affine
Grassmannian homology Schubert class factorization result of Magyar
\cite{M}.
\end{remark}

\begin{remark}
Let $f^\perp$ denote the linear operator adjoint to multiplication by a symmetric function $f$ under the Hall inner product.
One can get $k$-Schur functions straight from ordinary Schubert polynomials by directly substituting the ratio $(h_i^\perp \cdot
S_{R_m})/S_{R_m}$ of two Schur functions for each elementary symmetric polynomial $e_i(m) = e_i(x_1,x_2,\ldots,x_m)$, after the Schubert polynomial is written in terms of products of $e_i(m)$ (see Section \ref{ssec:qschub}): one may compute the image of $\Schub_w^q$ in
$\Z[h_1,\dotsc,h_{n-1}][s_{R_1}^{-1},\dotsc,s_{R_{n-1}}^{-1}]$ by
replacing each $E^q_i(m)$ by $(h_i^\perp \cdot S_{R_m})/S_{R_m}$ (see Proposition \ref{P:QSchurimage}).
However this is the same as taking the expansion of $\mathfrak{S}_w$
in the $e_{i_1}(1)e_{i_2}(2)\dotsm e_{i_{n-1}}(n-1)$ basis and
making the substitution $e_i(m)\mapsto (h_i^\perp \cdot
S_{R_m})/S_{R_m}$.
\end{remark}

\subsection{Further directions}

In future work, we plan to investigate the generalizations to equivariant (quantum) (co)homology and extensions to other Dynkin types.

A tantalizing open problem is to give a conceptual answer to the question: {\it does the Toda lattice know about Schubert
calculus?}

\medskip
{\bf Acknowledgments.}
We thank Takeshi Ikeda for interesting conversations which inspired this work.

\medskip

\section{Toda lattice and quantum cohomology of flag manifolds}

\subsection{Toda lattice}
The Toda lattice is the Hamiltonian system consisting of $n$ particles with positions $\tq_i$ and momenta $\tp_i$ and Hamiltonian
$$
H(\tp,\tq) = \sum_{i=1}^n \tp^2/2 + \sum_{i=1}^{n-1} e^{\tq_i - \tq_{i+1}}.
$$
The Toda lattice can be reformulated as the system of differential equations $dL/dt = [L,L_-]$ in terms of the Lax pair
$$
L(x,q) = \left(\begin{array}{ccccc}
x_1&-1&&\vphantom{\ddots} \\
q_1&x_2&-1&\vphantom{\ddots} \\
&q_2&x_3&\ddots \\
&&\ddots&\ddots&-1 \\
&&&q_{n-1}&x_n
\end{array}\right)
\qquad
L_{-} =\left(\begin{array}{ccccc}
0&0&&\vphantom{\ddots} \\
q_1&0&0&\vphantom{\ddots} \\
&q_2&0&\ddots \\
&&\ddots&\ddots&0 \\
&&&q_{n-1}&0
\end{array}\right)
$$
where the variables $x_i$ are multiples of $\tp_i$, and $q_i$ are multiples of $e^{\tq_i-\tq_{i+1}}$.  It follows from general theory that $H_k:= {\rm tr}(L^{k+1}/(k+1))$ gives a complete set of commuting Hamiltonians. Thus the Toda lattice is a completely integrable system.

\subsection{Givental and Kim's description of $QH^*(\Fl_n)$}

Let $$Y_0 = \{L(x,q) \mid \mbox{$L$ is nilpotent}\}$$ be the nilpotent Toda leaf.  This is the part of phase space where all Hamiltonians vanish.  Let
$$Y_0^\circ= \{L \in Y_0 \mid q_i \neq 0\}$$
be the part of the nilpotent Toda lattice where the quantum parameters are non-zero.

\begin{theorem}[\cite{GK,Kim}]\label{T:GK}
$$
QH^*(\Fl_n) \simeq \C[Y_0] = \C[x_1,\ldots,x_n,q_1,\ldots,q_{n-1}]/\langle H_k \rangle
$$
\end{theorem}

\subsection{Explicit Schubert representatives}\label{ssec:qschub}
Let $w \in S_n$ be a permutation, and let $s_{i_1} s_{i_2} \cdots s_{i_\ell}$ be a reduced decomposition of $w^{-1} w_0$, where $w_0 \in S_n$ is the longest permutation.  The {\it Schubert polynomial} $\Schub_w \in \Z[x_1,x_2,\ldots,x_{n-1}]$ is defined as
$$
\partial_{i_1} \partial_{i_2} \cdots \partial _{i_\ell}(x_1^{n-1} x_2^{n-2} \cdots x_{n-2})
$$
where $\partial_i$ denotes the {\it divided difference operator}
$$
(\partial_i f)(x_1,x_2,\ldots,x_{n-1}) = \frac{f(x_1,\ldots,x_{n-1}) - f(x_1,\ldots,x_{i+1},x_i,\ldots,x_{n-1})}{x_i - x_{i+1}}.
$$

Let $e_i(m) = e_i(x_1,x_2,\ldots,x_m)$ denote the elementary symmetric functions in $m$-variables.  Let $E^q_i(m)$ be the quantum analogue of the $i$-th elementary
symmetric polynomial in variables $x_1,x_2,\dotsc,x_m$. It is
defined by
\begin{equation}
  E^q_i(m) = E^q_i(m-1) + x_m E^q_{i-1}(m-1) + q_{m-1}
  E^q_{i-2}(m-2)
\end{equation}
where $E^q_i(m)=0$ if $i<0$ or $i>m$ and $E^q_0(m)=1$ for $m\ge0$.  Fomin, Gelfand, and Postnikov \cite{FGP} define the {\it quantum Schubert polynomial} $\Schub_w^q$
by expanding the ordinary Schubert polynomial $\Schub_w$ into the basis
$e_{i_1}(1)e_{i_2}(2)\dotsm e_{i_{n-1}}(n-1)$ where
$(i_1,\dotsc,i_{n-1})$ is a tuple of integers such that $0\le i_r
\le r-1$ for $1\le r\le n-1$ and then substituting the quantum elementary symmetric polynomial $E^q_i(m)$ for
each $e_i(m)$.

The quantum cohomology ring $QH^*(\Fl_n)$ has a $\C[q_1,\ldots,q_{n-1}]$-basis of quantum Schubert classes $\{\sigma^w \mid w \in S_n\}$ labeled by permutations.

\begin{theorem}[\cite{FGP}]
Under the isomorphism of Theorem \ref{T:GK} we have $\sigma^w \mapsto \Schub_w^q \mod \langle H_k \rangle$.
\end{theorem}

\section{Centralizer groups and homology of affine Grassmannian}
Let $\Gr = SL_n(\C((t)))/SL_n(\C[[t]])$ denote the affine Grassmannian of $G = SL(n)$.  Let $G^\vee = PGL(n)$ denote the Langlands dual of $G$.  Let $e = \sum_{i=1}^{n-1}e_i^\vee$ denote the principal nilpotent element
$$
e =
\left(\begin{array}{ccccc}
0&-1&&\vphantom{\ddots} \\
&0&-1&\vphantom{\ddots} \\
&&0&\ddots \\
&&&\ddots&-1 \\
&&&&0
\end{array}\right)
$$ in the Lie algebra $\geh^\vee$. Write
$$
X = G^\vee_e = \left\{\left( \begin{array}{cccccc} 1&h_1&h_2&h_3&\dotsm&h_{n-1} \\ 0&1&h_1&h_2&\dotsm&h_{n-2} \\[-1mm]
0&0&1&\ddots&\ddots&\vdots \\ 0&0&0&1&h_1&h_2 \\ 0&0&0&0&1&h_1\\ 0&0&0&0&0&1\end{array} \right)\right\} \subset G^\vee
$$
for the centralizer subgroup of $e$ in $G^\vee$.

The following result is due to Ginzburg \cite{Gin} and Peterson \cite{P}.
\begin{theorem}[\cite{Gin,P}]\label{T:GP} There are Hopf isomorphisms
$$
H_*(\Gr) \simeq \C[X] \simeq \Z[h_1,h_2,\ldots,h_{n-1}] = \Lambda_{(n)}
$$
\end{theorem}

\subsection{Explicit Schubert representatives}

The homology $H_*(\Gr, \C)$ has a $\C$-basis of affine Schubert classes $\{\xi_x \mid x \in \tS_n/S_n\}$ indexed by cosets of the symmetric group in the affine symmetric group $\tS_n$.  Recall that the {\it $k$-Schur functions} $s^{(k)}_\la$ \cite{LM,LLM} are labeled by $k$-bounded partitions: partitions $\la$ satisfying $\la_1 \leq k$.
The following theorem of the first author was conjectured by the second author.
\begin{theorem}[\cite{L}]\label{T:L}
Under the isomorphism of Theorem \ref{T:GP} we have $\xi_x \mapsto s_{b(x)}^{(k)}$, where the bijection $x \leftrightarrow b(x)$ between $\tS_n/S_n$ and $(n-1)$-bounded partitions is described in \cite{LLMS}.
\end{theorem}

For two partitions $\la,\mu$, we let $\la \cup \mu$ be the partition obtained by taking the union of parts of $\la$ and $\mu$.
The {\it $k$-rectangles} $R_i$ play a special role in the theory of $k$-Schur functions because of the following factorization result of Lapointe and Morse \cite{LM}:
\begin{theorem}\label{T:LM}
$$s^{(k)}_{\la \cup R_i} = s^{(k)}_\la s_{R_i}.$$
\end{theorem}

\section{Kostant's solution to the Toda lattice}
Since $\C[X] \simeq \C[h_1,h_2,\ldots,h_{n-1}]$, via the Jacobi-Trudi formula
$s_\la = \det(h_{\la_i-i+j})$ the Schur functions $s_{R_i}$ can be considered as polynomial functions on $X$.  We define
the Zariski-open set $X^\circ\subset X$ by
$$
X^\circ = \{g \in X \mid s_{R_i}(g) \neq 0 \mbox{ for $i \in [1,n-1]$}\}.
$$

Kostant \cite{Kos} solves  the nilpotent Toda lattice by
\begin{theorem}\label{T:K}
There is an isomorphism $\Psi:X^\circ \to Y_0^\circ$ such that the induced map
$\Psi^*:\C[Y_0^\circ] \to \C[X^\circ]$ is given by $\Psi^* = \Phi$.
\end{theorem}
Thus coordinates on $X$ can be considered as angle coordinates for the nilpotent Toda leaf.  The map $\Psi$ is constructed as follows: for $g \in X^\circ$, find a lower unitriangular matrix $n_-(g)$ so that $gn_-(g)$ has the form
$$
g n_-(g) = \left( \begin{array}{cccc} 0&0&0&* \\ 0&0&*&* \\
0&*&*&* \\ *&*&*&*\end{array} \right).
$$
Then $\Psi(g) = n_-^{-1}(g) e n_-(g)$.  The formula we give in Theorem \ref{T:main} for $\Phi$ is a symmetric function translation of Kostant's description of $\Psi(g)$.

\begin{remark}
In fact, $n_-(g)$ has $(i,j)$-th entry
$(-1)^{i-j} (e_{i-j}^\perp\cdot s_{R_j})/s_{R_j}$ for $i>j$, and $n_-(g)^{-1}$ is lower unitriangular with $(i,j)$
entry $(h_{i-j}^\perp \cdot s_{R_{i-1}})/s_{R_{i-1}}$ for $i>j$.
\end{remark}

Composing $\Psi^*$ with Theorems \ref{T:GK} and \ref{T:GP} we obtain an isomorphism $QH^*(\Fl_n)[q_i^{-1}] \cong H_*(\Gr)[s_{R_i}^{-1}]$.  We shall now discuss the behavior of Schubert classes under this isomorphism.

\section{Isomorphisms in terms of Schubert classes}
\label{S:iso}

\subsection{Peterson's isomorphism}
We first develop some notation allowing us to label affine Schubert classes with extended affine symmetric group elements.\footnote{We could alternatively work with $H_*(\Gr_{PGL(n)})$, but it is simpler to always use $\Gr = \Gr_{SL(n)}$ throughout.}  Our notation for affine Weyl groups mostly follows that in \cite{LS} (see also Appendix A).  For explicit affine symmetric group notation, we refer the reader to \cite{LLMS}.

Let $Q^\vee$ and $P^\vee$ denote the coroot lattice and coweight lattice of the root system $A_{n-1}$.  Let $\hS_n$ denote the extended affine symmetric group, so that $\hS_n \cong \Z/n\Z \ltimes \tS_n \cong W \ltimes P^\vee$.  Let $\hS_n^0$ and $\tS_n^0$ denote the minimum length coset representatives  in $\hS_n/S_n$  and $\tS_n/S_n$ respectively.  An affine permutation $x \in \hS_n$ can be thought of as a bijection $x: \Z \to \Z$ satisfying the periodicity condition $x(i+n) = x(i)+n$, and is determined by the window $[x(1),x(2),\ldots,x(n)]$.  The two affine permutations  $[x(1),x(2),\ldots,x(n)]$ and $[x(1)+n,x(2)+n,\ldots,x(n)+n]$ are considered identical.
Given $(\la_1,\la_2,\ldots,\la_n) \in P^\vee = \Z^n/(1,1,\ldots,1)$, the translation element $t_\la$ has window notation $t_\la = [1 + n\la_1,2 + n \la_2,\ldots,n+n\la_n]$.  Some distinguished elements in $P^\vee$ are the fundamental coweights $\omega_i^\vee = e_1+e_2+\cdots +e_i$, and the simple coroots $\alpha_i^\vee = e_i-e_{i+1}$, where $e_i \in \Z^n/(1,1,\ldots,1)$ denotes the standard basis vectors.
\par

For $x \in \hS_n$, we write $x = zy$ where $z \in \Z/n\Z$ and $y \in \tS_n^0$.  Then we define $\xi_x = \xi_y$ in $H_*(\Gr)$.

For $i \in [1,n-1]$ let $w_0^{\omega_i}$ denote the longest minimal length coset representative of $S_n/(S_i \times S_{n-i})$.  In one line notation, $w_0^{\omega_i}$ is given by $[n-i+1,n-i+2,\ldots,n,1,2,\ldots,n-i]$.  We then set $q_{\omega_i^\vee} = \sigma^{w_0^{\omega_i}}$.  With this definition one obtains $q_\lambda$ for each $\lambda \in P^\vee$, satisfying $q_{\lambda+\mu} = q_{\lambda} q_\mu$.    See Appendix \ref{S:extend} for further details.

In his study of geometric models for quantum cohomology of partial flag varieties, Peterson relates the quantum cohomology rings with the homology of affine Grassmanians.  Peterson's statement \cite{P} and our proof in \cite{LS} labels Schubert classes only by non-extended affine Weyl group elements.  In Appendix \ref{S:extend}, we show that the notation can be included to include extended affine Weyl group elements.

\begin{theorem}\label{T:P}
There is an isomorphism $QH^*(\Fl_n)[q_i^{-1}] \to H_*(\Gr)[\xi_{t_\lambda}^{-1} \mid \lambda \in P^\vee]
$ given by
$$
\sigma^w q_{\mu-\lambda} \longmapsto \xi_{wt_\mu} \xi_{t_\lambda}^{-1}
$$
where $\mu,\lambda \in P^\vee$ are antidominant.
\end{theorem}

\subsection{$k$-rectangles}
Define $d_i \in \tS_n^0$ to be the unique affine Grassmannian element in the same left $\Z/n\Z$ coset as $t_{-\omega_i^\vee}$.  Equivalently, $d_i$ is obtained by affine Dynkin rotation of $w_0^{\omega_i}$.  There is a bijection $w\mapsto b(w)$ \cite{LLMS} between $\tS_n^0$ and the set of $(n-1)$-bounded partitions.  Under this bijection we have $d_i \mapsto R_i$.

\begin{example} Let $n=5$ and $i=2$ so that $i^*:=n-i=3$.
Then $w_0^{\omega_3} = s_2s_1s_3s_2s_4s_3$ which in one-line notation is
$[3,4,5,1,2]$. and $d_2 = s_4s_3s_0s_4s_1s_0$.  The reduced word for $d_i$ is obtained by reading the entries from the French tableau of shape $R_i$ obtained by placing $c-r \mod
n$ into the cell in the $r$-th row and $c$-th column.
\begin{equation*}
\tableau[sby]{3&4\\4&0\\0&1}.
\end{equation*}
\end{example}

By Theorem \ref{T:L}, the isomorphism $H_*(\Gr_{SL_n}) \cong \Z[h_1,\dotsc,h_{n-1}]$ of Theorem \ref{T:GP} sends 
$\xi_w$ to
the $k$-Schur function $s_{b(w)}^{(n-1)}$. Therefore $\xi_{d_i} \mapsto s_{R_i}$
under the isomorphism of Theorem \ref{T:GP}.
This isomorphism extends to a localized isomorphism
\begin{align}\label{E:iso}
  H_*(\Gr_{SL_n})[\xi_{t_\la}^{-1}\mid \la\in \tQ]\cong
  \Z[h_1,\dotsc,h_{n-1}][s_{R_1}^{-1},\dotsc,s_{R_{n-1}}^{-1}]
\end{align}
under which $\xi_{t_{\omega_i^\vee}} =
\xi_{-t_{\omega_i^\vee}}^{-1}\mapsto 1/s_{R_i}$.  We denote by $\Phi'$
the isomorphism
\begin{align*}
QH^*(\Fl_n)[q_1^{-1},q_2^{-1},\ldots,q_{n-1}^{-1}] &\to
\Z[h_1,\dotsc,h_{n-1}][s_{R_1}^{-1},\dotsc,s_{R_{n-1}}^{-1}]
\end{align*}
given by composing Theorem \ref{T:P} with \eqref{E:iso}.
By the above discussion we have $\Phi'(q_{\omega_i^\vee}) = 1/s_{R_i}$.

Recall that $\Phi$ is defined before Theorem \ref{T:main}.
\begin{proposition} We have $\Phi=\Phi'$, where both are considered as algebra maps from $QH^*(\Fl_n)[q_i^{-1}]$ to $\Z[h_1,\dotsc,h_{n-1}][s_{R_i}^{-1}]$.
\end{proposition}
\begin{proof} It suffices to check equality on algebra generators
$q_i$ and $x_1+\dotsm+x_i$. For $q_i$, defining
$\omega_0^\vee=\omega_n^\vee=0$ we have $\alpha_i^\vee =
-\omega_{i+1}^\vee + 2 \omega_i^\vee-\omega_{i-1}^\vee$. Therefore
\begin{align}
  \Phi'(q_i) = \Phi'(q_{\alpha^\vee_i}) =
  \dfrac{s_{R_{i+1}}s_{R_{i-1}}}{s_{R_i}^2} = \Phi(q_i)
\end{align}
for $1\le i\le n-1$. We also have
\begin{align*}
  \Phi'(x_1+\dotsm+x_i) = \Phi'(E^q_1(i))= \dfrac{h_1^\perp R_i}{R_i}
  = \Phi(x_1+\dotsm+x_i)
\end{align*}
where we have used Proposition \ref{P:QSchurimage} for $\la$ a column of size
$i$.
\end{proof}

\subsection{Quantum Schur symmetric polynomials}

For a partition $\la=(\la_1,\dotsc,\la_m)$ define the quantum Schur
function
\begin{align} S^q_\la(x_1,\dotsc,x_m) = \det
E^q_{\la'_i-i+j}(x_1,\dotsc,x_m).
\end{align}
\begin{prop}\label{P:QSchurimage}
For $\la_1 \le n-m$, we have
\begin{align*}
\Phi'(S^q_\la(x_1,\dotsc,x_m)) = \dfrac{s_{\la'}^\perp
s_{R_m}}{s_{R_m}}.
\end{align*}

\end{prop}
\begin{proof}
Let $1\le m\le n-1$ and $\la$ a partition such that $\ell(\la)\le m$
and $\la_1\le n-m$, so that $\la \subset R_{n-m}$. Let $w_{\la,m}$
be the unique $m$-Grassmannian permutation in $S_n$ of shape $\la$:
in one line notation, $w_{\la,m}$ starts with
$1+\la_m,2+\la_{m-1},\dotsc, m+\la_1$, and ends with the
complementary numbers in increasing order. It is known that
$s_\la(x_1,\dotsc,x_m)=\det(e_{\la'_i-i+j})$ is the ordinary
Schubert class of the $m$-Grassmannian permutation $w_{\la,m}$ in
$H^*(\Gr(m,n))$. Quantizing this relation (that is, using
\cite{FGP}) yields $\sigma^{w_\la,m} = \det (E^q_{\la'_i-i+j}) =
S^q_\la(m)$. We have
\begin{align*}
\Phi'(S^q_\la(m)) = \Phi'(\sigma^{w_{\la,m}}) = \xi_{w_{\la,m}
t_{-\omega_m^\vee}} \xi_{t_{-\omega_m^\vee}}^{-1}.
\end{align*}
By direct computation, $w_{\la,m}t_{-\omega_m^\vee} =
b^{-1}(\la^\vee)$ where $\la^\vee$ is the partition obtained by
taking the complement of $\la$ in $R_{n-m}$ and then taking the
transpose. Therefore
\begin{align*}
\Phi'(S^q_\la(m)) &= \dfrac{s_{\la^\vee}}{s_{R_m}}.
\end{align*}\end{proof}

\section{Explicit computation of $\la(w)$}\label{S:la}
We now describe the map $w\mapsto \la(w)$ of Theorem \ref{T:main} explicitly.  For simplicity of notation, we first assume that $w \in S_n$ satisfies $w(1) = 1$.

Let $c_i =  s_{n-i}\dotsm s_{n-2}s_{n-1} \in S_n$ for $1\le i\le n-1$.  Then there is a unique sequence $(m_1,m_2,\dotsc,m_{n-1})$ of integers such that $0\le m_1 \le n-2$,
$0\le m_2 \le n-3$, $\dotsc$, $0\le m_{n-2} \le 1$ such that
\begin{align*}
  w = c_{n-2}^{m_1} \dotsm  c_2^{m_{n-3}} c_1^{m_{n-2}}.
\end{align*}
Then $\la(w)$ is the $(n-1)$-irreducible bounded partition $\la$ with $m_i$ equal to the number of parts of size $i$.  The following example shows how to obtain $m_i$ from $w$ algorithmically.  Another description of $w$ in terms of the inversion set of $w$ is given in the proof of Lemma \ref{L:law}.

\begin{example} Let $n=6$ and $\la = (4,3,2,2,2,1,1)$ so that $m_1=2$, $m_2=3$, $m_3=1$, $m_4=1$.
We have
\begin{align*}
  w = (s_2s_3s_4s_5)^2(s_3s_4s_5)^3(s_4s_5)^1(s_5)^1.
\end{align*}
In one line notation $w=[1,4,3,6,5,2]$. To go from $w$ to $\la(w)$,
we start with $[1,2,3,4,5,6]$ and must obtain $4$ in position $2$. Therefore we must
left circular shift twice in the last five positions, that is, $m_1=2$. We obtain $[1,4,5,6,2,3]$.
Next we must get $3$ into position $3$, which requires three left circular shifts in the last four positions,
that is, $m_2=3$. We obtain $[1,4,3,5,6,2]$. To get $6$ to position $4$ we need $m_3=1$ and we obtain $[1,4,3,6,2,5]$.
Finally to get $5$ to position $5$ we need $m_4=1$ and we obtain $[1,4,3,6,5,2]$ as required.
\end{example}

\begin{lem}\label{L:law}
The element $y = w \prod_{i \in \Des(w)} t_{-\omega_i^\vee}$ lies in $\hS_n^0$, and modulo
conjugation by an affine Dynkin diagram rotation, is equal to $x \in \tS_n^0$ where $b(x) = \la(w)$.
\end{lem}
\begin{proof}
For $1 \leq i \leq n$, let $\des_i(w)$ denote the number of descents of $w$ which lie before $i$.
In one line notation, one has $y(i) = w(i) - \des_i(w)n$.  Since $y \in \hS_n^0$ if and only if $y(1) < y(2) < \cdots < y(n)$, the first claim follows.

To obtain $b(x)$, we use \cite[Proposition 8.15]{LLMS}.  In the current setting, the proposition says that we must show
\begin{equation}\label{E:lar}
\la'_r = \#\{j < r \mid y(j) > y(r)\}
\end{equation}
where we think of $y$ as a bijection from $\Z$ to $\Z$ when evaluating $y(j)$.  For $1 \leq i < j \leq n$, define $a(i,j) := \des_j(w) - \des_i(w) - \chi(w(i) > w(j))$.  The RHS of \eqref{E:lar} is then equal to $\sum_{j > r} a(r,j)$.

It is straightforward to prove by induction the following characterization of the permutations $w = c_{n-2}^{m_1} \dotsm  c_{n-r}^{m_{r-1}}$: these are exactly the permutations satisfying $w(r) < w(r+1) < \cdots < w(s) > w(s+1) < w(s+2) < \cdots <w(n) < w(r)$.  In other words, the word $w(r)w(r+1) \cdots w(s)$ has one cyclic descent.  Using this, it is easy to see that successive multiplication by $c_{n-r}$ increases $\sum_{j> r'} a(r',j)$ by 1 for each $r' \in [1,r-1]$.  This establishes \eqref{E:lar}.
\end{proof}

To complete the description of $\la(w)$, we remove the condition $w(1) = 1$.
For $w\in S_n$ in one line notation $w=[w(1),w(2),\dotsc,w(n)]$,
define $w'=[w(1)-1\mod n,w(2)-1\mod n,\dotsc,w(n)-1\mod n]$ where
the $\mod n$ function takes values in $\{1,2,\dotsc,n\}$.  The following result includes a property of quantum Schubert classes which may be new.

\begin{prop}\label{P:cyclic}
We have $\la(w') = \la(w)$ and
\begin{align*}
q_{\omega^\vee_{w^{-1}(1)-1}}   \sigma^{w'} = q_{\omega^\vee_{w^{-1}(1)}} \sigma^w
\end{align*}
where $q_{\omega^\vee_0}=q_{\omega^\vee_n}=1$.
\end{prop}
\begin{proof}
One has $w' =
s_{n-1} s_{n-2} \dotsm s_2 s_1 w$. Now we have $s_{n-1}\dotsm s_2
s_1 = w_0^{\omega_1} = \tau_1 t_{-\omega_1^\vee}$. We have
\begin{align*}
  \Phi(\sigma^{w'}) = \Phi(\sigma^{w_0^{\omega_1^\vee} w}) = \xi_{w_0^{\omega_1^\vee} w t_\beta} \xi_{t_\beta}^{-1}
\end{align*}
for sufficiently antidominant $\beta\in \tQ$. We compute
\begin{align*}
  w_0^{\omega_1^\vee} w t_\beta =
  w_0^{\omega_1^\vee} w t_{-w^{-1} \cdot\omega_1^\vee} t_{\beta
  +w^{-1} \cdot\omega_1^\vee}
  = w_0^{\omega_1^\vee}  t_{-\omega_1^\vee} w t_{\beta
  +w^{-1} \cdot\omega_1^\vee}
  = \tau_1 w t_{\beta+w^{-1}\cdot \omega_1^\vee}.
\end{align*}
It follows that
\begin{align*}
\Phi(\sigma^{w'}) &= \Phi(\sigma^w)
\dfrac{s_{R_{w^{-1}(1)-1}}}{s_{R_{w^{-1}(1)}}}.
\end{align*}
Now $\Des(w')$ is obtained from $\Des(w)$ by removing $w^{-1}(1)-1$
and adding $w^{-1}(1)$, so the result follows.
\end{proof}

\begin{equation*}
\begin{array}{|c|l|} \hline
w & \la(w) \\ \hline
(1,2) & () \\ \hline
\end{array}
\quad
\begin{array}{|c|l|} \hline
w & \la(w) \\ \hline
(1,2,3) & () \\
(1,3,2) & (1) \\ \hline
\end{array}
\quad
\begin{array}{|c|l|} \hline
w & \la(w) \\ \hline
(1,2,3,4) & () \\
(1,2,4,3) & (2) \\
(1,3,2,4) & (2,1) \\
(1,3,4,2) & (1) \\
(1,4,2,3) & (1,1) \\
(1,4,3,2) & (2,1,1) \\ \hline
\end{array}
\end{equation*}

\begin{equation*}
\begin{array}{|c|l|} \hline
w & \la(w) \\ \hline
(1,2,3,4,5) & () \\
(1,2,3,5,4) & (3) \\
(1,2,4,3,5) & (3,2) \\
(1,2,4,5,3) & (2) \\
(1,2,5,3,4) & (2,2) \\
(1,2,5,4,3) & (3,2,2) \\
(1,3,2,4,5) & (2,2,1) \\
(1,3,2,5,4) & (3,2,2,1) \\
(1,3,4,2,5) & (3,1) \\
(1,3,4,5,2) & (1) \\
(1,3,5,2,4) & (2,1) \\
(1,3,5,4,2) & (3,2,1) \\
(1,4,2,3,5) & (2,1,1) \\
(1,4,2,5,3) & (3,2,1,1) \\
(1,4,3,2,5) & (3,2,2,1,1) \\
(1,4,3,5,2) & (2,2,1,1) \\
(1,4,5,2,3) & (1,1) \\
(1,4,5,3,2) & (3,1,1) \\
(1,5,2,3,4) & (1,1,1) \\
(1,5,2,4,3) & (3,1,1,1) \\
(1,5,3,2,4) & (3,2,1,1,1) \\
(1,5,3,4,2) & (2,1,1,1) \\
(1,5,4,2,3) & (2,2,1,1,1) \\
(1,5,4,3,2) & (3,2,2,1,1,1) \\ \hline
\end{array}
\end{equation*}

%


\appendix
\section{Extending the Peterson isomorphism}\label{S:extend}
In this appendix we work in the setting of an arbitrary Weyl group $W$ of a simply-connected algebraic group $G$.  Our notation follows that of \cite{LS}.  Let $I=\{1,2,\dotsc,r\}$ and
\begin{align*}
P^\vee &= \bigoplus_{i=1}^r \Z \omega_i^\vee \quad&\tP &= \{\mu\in
P^\vee\mid \text{$\ip{\mu}{\alpha_i} \le 0$ for all $i\in I$} \}
 \\
Q^\vee &= \bigoplus_{i=1}^r \Z \alpha_i^\vee
 &
\tQ &= \tP \cap Q^\vee.
\end{align*}
Let $\{\sigma^w\in QH^*(G/B) \mid w\in W\}$ denote the Schubert
basis of the small quantum cohomology ring of $G/B$. Let $\{\xi_x\in
H_*(\Gr_G) \mid x\in W_\af^0\}$ be the Schubert basis of the
homology ring of the affine Grassmannian of $G$.

\begin{lem} \label{L:coset}
Let $w\in W$ and $\mu\in Q^\vee$. Then $w t_\mu\in W_\af^0$ if and
only if $\mu\in \tQ$ and whenever $wr_i<w$ for some $i\in I$ we have
$\ip{\mu}{\alpha_i} < 0$. In particular, letting $\rho^\vee
=\sum_{i\in I} \omega^\vee_i$, for any $w\in W$ we have $w t_\mu\in
W_\af^0$ for all $\mu\in Q^\vee$ such that $\mu + 2\rho \in \tQ$.
\end{lem}

Peterson \cite{P} (see also \cite{LS}) defined a ring isomorphism
\begin{equation}
\label{E:Petiso}
\begin{split}
QH^*(G/B)[q_1^{-1},\dotsc,q_r^{-1}] &\cong
H_*(\Gr_G)[\xi_{t_\mu}^{-1}\mid \mu\in \tQ] \\
\sigma^w q_{\la-\mu} &\mapsto \xi_{w t_\la} \xi_{t_\mu}^{-1}
\end{split}
\end{equation}
for $w\in W$ and $\la,\mu\in \tQ$ such that $w t_\mu\in W_\af^0$.  We wish to give a more precise description of the denominators that
occur in the right hand side of \eqref{E:Petiso}.

Let $\Aut(I_\af)$ denote the group of automorphisms of the affine
Dynkin diagram. Let $I^s = \Aut(I_\af)\cdot \{0\}$ be the set of
special nodes where $0\in I_\af$ is the distinguished affine node.
There is a bijection $I^s\cong P^\vee/Q^\vee$ such that $i\mapsto
-\omega_i^\vee+Q^\vee$ where $\omega_i^\vee$ is the fundamental
coroot for $i\in I^s\setminus \{0\}$ and $\omega_0^\vee = 0$. For
each $i\in I^s$, subtraction by $\omega_i^\vee+Q^\vee$ induces a
permutation of $I^s$ denoted $\tau_i$ which extends uniquely to an
element $\tau_i\in \Aut(I_\af)$, which is called the special
automorphism associated with $i\in I^s$. It satisfies $\tau_i(i)=0$.
There is a group monomorphism $P^\vee/Q^\vee \to \Aut(I_\af)$ such
that $-\omega_i^\vee+Q^\vee\mapsto \tau_i$ for $i\in I^s$. We denote
the image of this map by $\Aut^s(I_\af)$, the subgroup of special
automorphisms. We write $i\mapsto i^*$ for the element of
$\Aut(I_\af)$ such that $0^*=0$ and $w_0 r_i w_0 = r_{i^*}$ for
$i\in I$. Equivalently $-w_0\cdot\alpha_i = \alpha_{i^*}$ or
$-w_0\cdot \omega_i^\vee = \omega^\vee_{i^*}$, or
$-\omega_i^\vee+Q^\vee = \omega_{i^*}^\vee+Q^\vee$ for $i\in I^s$.

Let $W_\ex \cong W \ltimes P^\vee \cong \Aut^s(I_\af) \ltimes W_\af$
be the extended affine Weyl group. For $w\in W$ and $\mu\in P^\vee$
we have $w t_\mu w^{-1} = t_{w\cdot \mu}$ and for $z\in
\Aut^s(I_\af)$ and $i\in I_\af$ we have $z r_i z^{-1} = r_{z(i)}$.
We use the level zero action of $W_\ex$ on $P^\vee$ given by $u
t_\la\cdot \mu = u\cdot \mu$ for $\la,\mu\in P^\vee$ and $u\in W$.
Then we have $z t_\mu z^{-1} = t_{z\cdot\mu}$ for all $z\in W_\ex$
and $\mu\in P^\vee$. With this notation we have
\begin{align}\label{E:tau}
\tau_i = w_0^{\omega_i} t_{-\omega_i^\vee}\qquad \text{for $i\in
I^s$}
\end{align}
where $w_0^{\omega_i} \in W$ is the shortest element in the coset
$w_0 W_{\omega_i}$ where $W_{\omega_i}$ is the stabilizer of
$\omega_i$ and $w_0\in W$ is the longest element. In particular if
$i=0$ then $w_0^{\omega_0}=\id$.

For $i\in I$, define $d_i\in W_\af$ to be the unique element such that
$$\Aut^s(I_\af) \,t_{-\omega_i^\vee} = \Aut^s(I_\af)\, d_i.$$

\begin{lem}\label{L:d} $d_i\in W_\af^0$.
If $i\in I^s$,
\begin{align} \label{E:dspecial}
  d_i = \tau_{i^*} \,w_0^{\omega_{i^*}} \,\tau_{i^*}^{-1}
  = \tau_{i^*} t_{-\omega_i^\vee}
  = w_0^{\omega_{i^*}} t_{-\omega_i^\vee-\omega_{i^*}^\vee}.
\end{align}
For general $i\in I$, let $j\in I^s$ be such that $\omega_i \equiv \omega_j \mod Q^\vee$. Then
\begin{align} \label{E:d}
  d_i = d_j \,t_{w_0^{\omega_{j^*}}\cdot(-\omega_i^\vee+\omega_j^\vee)}.
\end{align}
\end{lem}
\begin{proof} Suppose first that $i\in I^s$. Then
\begin{align*}
  t_{-\omega_i^\vee} &= (w_0^{\omega_i})^{-1} \tau_i \\
  &= \tau_i \tau_i^{-1} (w_0^{\omega_i})^{-1} \tau_i \\
  &= \tau_i (\tau_{i^*} w_0^{\omega_{i^*}} \tau_{i^*}^{-1}).
\end{align*}
By definition $d_i$ satisfies \eqref{E:dspecial}.
We now check that $d_i\in W_\af^0$.
Using the automorphism $\tau_{i^*}$ we see that this holds if and only if
$w_0^{\omega_{i^*}}$ is lengthened by right multiplication by $r_k$ for all $k\ne i^*$.
This is true for $k=0$ since $w_0^{\omega_{i^*}}\in W$ and true for $k\in I\setminus \{i^*\}$
by definition. So $d_i\in W_\af^0$.

Let $i\in I$ be general with $j\in I^s$ and $\beta\in Q^\vee$ such that
$\omega_i^\vee=\omega_j^\vee+\beta$. We have
\begin{align*}
  t_{-\omega_i^\vee} &= t_\beta \,t_{-\omega_j^\vee} \\
  &= t_\beta \,\tau_j \,d_j \\
  &= \tau_j \,t_{\tau_{j^*}\cdot \beta} \,d_j \\
  &= \tau_j \,t_{w_0^{\omega_{j^*}}\cdot\beta}\, d_j.
\end{align*}
By definition \eqref{E:d} holds.
We have $d_i\in W_\af^0$ essentially because the factorization
$(t_\beta)(t_{-\omega_j^\vee})$ is length-additive.
\end{proof}

Say that $w\in W_\af^0$ is $i$-reducible for $i\in I$, if
$\ell(w)=\ell(wd_i)+\ell(d_i)$. Say that $w\in W_\af^0$ is
irreducible if it is not $i$-reducible for any $i\in I$.

\begin{prop} \label{P:fracs} $H_*(\Gr_G)[\xi_{t_\mu}^{-1}\mid \mu\in \tQ]$ has $\Z$-basis given by
$\xi_w \prod_{i\in I} \xi_{d_i}^{e_i}$ for $w\in W_\af^0$
irreducible and $e_i\in \Z$ for $i\in I$.
\end{prop}

\begin{lem} \label{L:factor} Suppose $w\in W_\af^0$ is
$i$-reducible. Let $j\in I^s$ be such that $\omega_i^\vee + Q^\vee =
\omega_j^\vee + Q^\vee$. Then $\tau_j wd_i \tau_j^{-1}\in W_\af^0$ and
\begin{align*}
  \xi_w = \xi_{\tau_j wd_i \tau_j^{-1}} \xi_{d_i}.
\end{align*}
\end{lem}

Magyar has a criterion \cite{M} for finding the largest product of
elements $\xi_{d_i}$ that factor out of $\xi_w$ for $w\in W_\af^0$.

Define $W_\ex^0$ be the set of elements of minimum length in their
cosets in $W_\ex/W$. Then $W_\ex^0 =\Aut^s(I_\af) \ltimes W_\af^0$.
Note that $W_\ex^0\cap P^\vee = \tP$. For $x\in W_\ex^0$ let $x=zy$
where $z\in \Aut^s(I_\af)$ and $y\in W_\af^0$. Then define $\xi_x\in
H_*(\Gr_G)$ by
\begin{align}
  \xi_x = \xi_y.
\end{align}
In particular, for all $i\in I^s$,
\begin{align}
  \xi_{t_{-\omega_i^\vee}} = \xi_{\tau_i d_i} = \xi_{d_i}.
\end{align}
For $\mu\in P^\vee$ let $\la\in Q^\vee$ and $i\in I^s$ be the unique
elements such that $\mu=\omega_i^\vee+\la$. Then define $q_\mu\in
QH^*(G/B)$ by
\begin{align}
  q_\mu = q_\la \, q_{\omega_i^\vee} = q_\la\,
  \sigma^{{w_0}^{\omega_i}}.
\end{align}
One may show that $q_\la q_\mu = q_{\la+\mu}$ for all $\la,\mu\in
P^\vee$. We may extend the notation of Peterson's isomorphism by
writing
\begin{align}
  \sigma^w q_{\mu-\la} \mapsto \xi_{w t_\mu} \xi_{t_\la}^{-1}
\end{align}
where $\la,\mu\in \tP$ are such that $wt_\mu\in W_\ex^0$.


\end{document}